\documentclass[12pt]{article}
\usepackage{amsfonts}
\usepackage{amssymb}
\usepackage{graphicx}
\usepackage{amsmath}

\newtheorem{theorem}{Theorem}
\newtheorem{acknowledgement}[theorem]{Acknowledgement}

\newtheorem{corollary}[theorem]{Corollary}

\newtheorem{example}[theorem]{Example}

\newtheorem{lemma}[theorem]{Lemma}

\newtheorem{proposition}[theorem]{Proposition}
\newtheorem{remark}[theorem]{Remark}

\newenvironment{proof}[1][Proof]{\textbf{#1.} }{\ \rule{0.5em}{0.5em}}
\input{tcilatex}

\begin{document}

\title{\textbf{A Causal Construction of Diffusion Processes }}
\author{Tadeusz \ BANEK \\
Dept. of Quantitative Method, Faculty of Management\\
Tech. Univ. of Lublin, Nadbystrzycka 38, 20-618 Lublin, Poland, \\
e-mail; t.banek@pollub.pl}
\maketitle
\date{}

\begin{abstract}
A simple nonlinear integral equation for Ito's map is obtained. Although, it
does not include stochastic integrals, it does give causal construction of
diffusion processes which can be easily implemented by iteration systems.
Applications in financial modelling and extension to fBm are discussed.

\textbf{key words: }diffusion processes, fBm, translation of Wiener
processes, Girsanov theorem
\end{abstract}

\section{Introduction}

Diffusions are an important class of stochastic processes. They are Markov,
and have continuous trajectories. There are extensive, competent historical
surveys of the topic by D.W. Stroock given in [3, 5, 6], and we recommend
Stroock's discussion to the interested reader. Here, we shall point out only
a main stages. Historically, the first construction was given by A.N.
Kolmogorov in his 1931 famous paper [1], and, since then, a problem of
constructing diffusions in $\mathbb{R}^{n}$, having differential operators
as generators, and no barriers, is known as the Kolmogorov problem. We shall
restrict our attention to the later class and call it (for short), K$-$
diffusions. The second construction of K$-$ diffusions was given by K. Ito
in [2] (and it is called Ito diffusions too). The theory of
ordinary-stochastic equations by H. Sussman [7] and H. Doss [8], and some
modification (see [10]), may be regarded as a deterministic variant of Ito's
theory. The third is known as a solution of D.W. Stroock and S.R.S. Varadhan
martingale problem (see [3]). The fourth is given by the Isobe-Sato formula
[4], which gives Wiener-Ito integrals for chaos decomposition of K$-$
diffusions.

In this paper, we propose a new, pathwise variant of Ito's construction of K$%
-$ diffusions. Although the construction uses a Wiener process (Ito's idea),
it does not involve Ito's integrals. It consists in:

(a) Solving a nonlinear, deterministic, Volterra type integral equation%
\begin{equation}
c\left( w\left( t\right) -\int_{o}^{t}\varkappa \left( x\left( s\right)
\right) ds\right) =x\left( t\right)  \label{0}
\end{equation}%
where $w\in C_{T}\triangleq C\left( \left[ 0,T\right] ;\mathbb{R}\right) $, $%
c$ and $\varkappa $ are ordinary scalar function to be specify later. Under
mild assumptions (\ref{0}) can be solved pathwise and nonanticipative, i.e.,
for any $w$, $v\in C_{T}$ one finds $x_{w}$, $x_{v}\in C_{T}$, such that
restrictions $x_{w}$, $x_{v}$, on $\left[ 0,t\right] $ coincide, if $w\left(
s\right) =v\left( s\right) $, $s\in \left[ 0,t\right] $.

(b) Forming a map $X_{t}\left( w\right) :\left[ 0,T\right] \times
C_{T}\rightarrow \mathbb{R}$, such that $X_{t}\left( w\right) =x_{w}\left(
t\right) $. Hence, $X\left( w\right) $ belongs to the space $\mathfrak{G}%
\left( C_{T}\right) $ of all nonanticipative mappings from $C_{T}$ to $C_{T}$%
, and it is a fix point of the operator $\mathfrak{L}:\mathfrak{G}\left(
C_{T}\right) \rightarrow \mathfrak{G}\left( C_{T}\right) $, defined by%
\begin{equation}
\mathfrak{L}\left( X\left( w\right) \right) \left( t\right) =c\left( w\left(
t\right) -\int_{o}^{t}\varkappa \left( X\left( w\right) \left( s\right)
\right) ds\right)  \label{1}
\end{equation}%
where we adopt the convention $X\left( w\right) \left( t\right) =X_{t}\left(
w\right) $.

(c) Showing that $X_{t}\left( w\right) $ is a K$-$ diffusions assuming that $%
w\left( t\right) $, $t\in \left[ 0,T\right] $ is\ a Wiener process.

(d) Proving, it is true in the opposite direction as well, i.e., if $%
X_{t}\left( w\right) $ is a K$-$ diffusion, then it is a fix point of $%
\mathfrak{L}$.

It is instructive, to compare an intuitive picture behind Ito's theory
(here, we again recommend Stroock [3, 5, 6]), with the picture of K$-$
diffusions as suggested by (\ref{0}). In the first picture, infinitesimal
increments of K$-$ diffusions are resulting from combined effects of two
forces: a deterministic drift and random (Gaussian) fluctuations. Since
combination here means a sum, the both forces (deterministic and random)
have the same status in creation of K$-$ diffusions. However, (\ref{0})
suggests other picture, or looking from cybernetic perspective, better to
say a \textquotedblleft behavior\textquotedblright . Namely, $x_{w}$ follows 
$w$, what is easily visible on the diagram below%
\begin{equation*}
\left[ 
\begin{array}{c}
w\longrightarrow ^{\left( +\right) }\otimes \text{\ }\longrightarrow \left[
c\left( \cdot \right) \right] \longrightarrow \downarrow \longrightarrow 
\text{ }x_{w} \\ 
\text{ \ \ }^{\left( -\right) }\uparrow \longleftarrow \left[ \int \right] 
\text{ }\longleftarrow \left[ \varkappa \left( \cdot \right) \right]%
\end{array}%
\right]
\end{equation*}%
which explain the idea of simple iteration system which works according to (%
\ref{0}). With $y\left( t\right) \triangleq \int_{0}^{t}\varkappa \left(
x\left( s\right) \right) ds$, this behavior is even more explicit%
\begin{equation*}
\frac{d}{dt}y\left( t\right) =\varkappa \circ c\left( w\left( t\right)
-y\left( t\right) \right)
\end{equation*}%
Hence $y_{w}\left( t\right) \triangleq \int_{0}^{t}\varkappa \left(
x_{w}\left( s\right) \right) ds$ follows $w$ with the speed equals the image
of the difference $w\left( t\right) -y_{w}\left( t\right) $ under $\varkappa
\circ c$. Thus, in this picture we have pure \textbf{deterministic}
mechanism, expressed in the terms of $\varkappa \circ c$ composition, which
forces $y_{w}$ to follow a \textbf{random }path $w$. Even more, a rule of
producing actions according to the current errors is known in Automatic
Control as a classical \textbf{feedback rule}, which in turns, is the most
transparent idea of \textbf{Cybernetics}. Is there any Variational Principle
responsible for this rule is an open question.

The paper is organized as follows. In a preliminary section we state an
auxiliary result on (\ref{0}). In the next section we prove an equivalence
theorem, which is the main result of this paper. Several corollaries are
also included. Indication for financial mathematics is discussed next.
Finally, a partial extension to fBm is included in the last section.

\section{Preliminaries}

We state here the following

\begin{lemma}
Assume $c:\mathbb{R\rightarrow R}$ is locally Lipschitz and $\varkappa :%
\mathbb{R\rightarrow R}$ measurable and bounded. Then, (a) for any $w\in
C_{T}$, there exists a unique $x_{w}\in C_{T}$ satisfying (\ref{0}), (b) for
any $w\in C_{T}$ and any $\xi \in C_{T}$, a sequence of successive
approximation%
\begin{eqnarray*}
x_{0} &=&\xi \text{, \ }x_{n+1}=\Phi _{w}\left( x_{n}\right) \\
\Phi _{w}\left( x\right) \left( t\right) &\triangleq &c\left( w\left(
t\right) -\int_{o}^{t}\varkappa \left( x\left( s\right) \right) ds\right)
\end{eqnarray*}%
is convergent in any norm $\left\Vert \cdot \right\Vert _{\lambda }$, $%
\lambda \geq 0$, to $x_{w}$, where $\left\Vert x\right\Vert _{\lambda }=\max
\left\{ e^{-\lambda t}\left\vert x\left( t\right) \right\vert ;0\leq t\leq
T\right\} $, (c) a mapping $C_{T}\ni w\mapsto X\left( w\right) \triangleq
x_{w}\in C_{T}$ is locally Lipschitz (in any $\left\Vert \cdot \right\Vert
_{\lambda }$, $\lambda \geq 0$), and nonanticipating.
\end{lemma}

\begin{proof}
The proof consists in two steps. In the first, one can show (a),(b),(c) hold
when $c$ is globally Lipschitz. In the second, one can apply a method of
continuation in the locally Lipschitz case. The proof is standard hence it
is omitted.
\end{proof}

\section{Equivalence theorem}

Let $g\in C^{1}\left( \Delta \right) $, $\Delta \subset \mathbb{R}$ and $%
f\in C\left( \mathbb{R}\right) $. Define two functions: 
\begin{equation}
c^{\prime }\left( x\right) =g\left( c\left( x\right) \right)  \label{2}
\end{equation}%
and%
\begin{equation}
\varkappa \left( x\right) =\frac{g^{\prime }\left( x\right) }{2}-\frac{%
f\left( x\right) }{g\left( x\right) }  \label{3}
\end{equation}

\begin{example}
(a) Let $g\left( x\right) =\sqrt{1+x^{2}}$. Then $c\left( x\right) =\sinh
\left( a+x\right) $. For an arbitrary $\phi \in C\left( \mathbb{R}\right) $,
set $f\left( x\right) =\frac{x}{2}-\phi \left( x\right) \sqrt{1+x^{2}}$,
then $\varkappa \left( x\right) =\phi \left( x\right) $. (b) Let $g\left(
x\right) =\left\vert x\right\vert ^{\alpha },\left\vert \alpha \right\vert
<1 $. Then for $a\in \mathbb{R}$, we have $c\left( x\right) =\left[
sign\left( a+x\right) \right] \left[ \left( 1-\alpha \right) \left\vert
a+x\right\vert \right] ^{1/1-\alpha }$. For $\phi \in C\left( \mathbb{R}%
\right) $, set $f\left( x\right) =\frac{\alpha }{2}sign\left( x\right)
\left\vert x\right\vert ^{2\alpha -1}-\left\vert x\right\vert ^{\alpha }\phi
\left( x\right) $, then $\varkappa \left( x\right) =\phi \left( x\right) $.
\end{example}

\begin{theorem}
Assume $c:\mathbb{R\rightarrow R}$, $\varkappa :\mathbb{R\rightarrow R}$
satisfy (\ref{2})(\ref{3}) and $\varkappa $ is bounded. If $w\left( t\right) 
$, $t\in \left[ 0,T\right] $ is\ a Wiener process on $\left( \Omega ,%
\mathfrak{F},\mathbb{P}\right) $, then the mapping $\left[ 0,T\right] \times
C_{T}\ni \left( t,w\right) \mapsto X_{t}\left( w\right) \in \mathbb{R}$
satisfies the equation%
\begin{equation}
c\left( w\left( t\right) -\int_{0}^{t}\varkappa \left( X_{s}\left( w\right)
\right) ds\right) =X_{t}\left( w\right)  \label{4}
\end{equation}%
$\mathbb{P-}$ a.s., iff solves (strongly) Ito's differential equation%
\begin{eqnarray}
dx\left( t\right) &=&f\left( x\left( t\right) \right) dt+g\left( x\left(
t\right) \right) dw\left( t\right)  \label{5a} \\
x\left( 0\right) &=&c\left( 0\right)  \label{5b}
\end{eqnarray}%
$\mathbb{P-}$ a.s.
\end{theorem}

\begin{proof}
Assume that $X_{t}\left( w\right) $ solves (\ref{4}), and denote%
\begin{equation*}
\widetilde{w}\left( t\right) \triangleq w\left( t\right)
-\int_{o}^{t}\varkappa \left( X_{s}\left( w\right) \right)
\end{equation*}%
From Ito's formula and (\ref{2})(\ref{3}) we get%
\begin{eqnarray}
&&dc\left( \widetilde{w}\left( t\right) \right)  \label{6} \\
&=&\left[ \frac{1}{2}c^{\prime \prime }\left( \widetilde{w}\left( t\right)
\right) -c^{\prime }\left( \widetilde{w}\left( t\right) \right) \varkappa
\left( X_{t}\left( w\right) \right) \right] dt+c^{\prime }\left( \widetilde{w%
}\left( t\right) \right) dw\left( t\right)  \notag \\
&=&\left[ \frac{1}{2}g\left( c\left( \widetilde{w}\left( t\right) \right)
\right) g^{\prime }\left( c\left( \widetilde{w}\left( t\right) \right)
\right) -g\left( c\left( \widetilde{w}\left( t\right) \right) \right)
\varkappa \left( X_{t}\left( w\right) \right) \right] dt+g\left( c\left( 
\widetilde{w}\left( t\right) \right) \right) dw\left( t\right)  \notag \\
&=&g\left( c\left( \widetilde{w}\left( t\right) \right) \right) \left\{ 
\frac{1}{2}g^{\prime }\left( c\left( \widetilde{w}\left( t\right) \right)
\right) -\left[ \frac{g^{\prime }\left( X_{t}\left( w\right) \right) }{2}-%
\frac{f\left( X_{t}\left( w\right) \right) }{g\left( X_{t}\left( w\right)
\right) }\right] \right\} dt  \notag \\
&&+g\left( c\left( \widetilde{w}\left( t\right) \right) \right) dw\left(
t\right)  \notag
\end{eqnarray}%
Since (by the assumption) 
\begin{equation*}
X_{t}\left( w\right) =c\left( \widetilde{w}\left( t\right) \right)
\end{equation*}%
thus the RHS of (\ref{6}) equals%
\begin{equation*}
=f\left( X_{t}\left( w\right) \right) dt+g\left( X_{t}\left( w\right)
\right) dw\left( t\right)
\end{equation*}%
Hence $X_{t}\left( w\right) $ solves (\ref{5a}),(\ref{5b})\ since $%
X_{0}\left( w\right) =c\left( 0\right) $).

Now in the reverse direction. Let $X_{t}\left( w\right) $, $X_{0}\left(
w\right) =c\left( 0\right) $, solves strongly (\ref{5a}),(\ref{5b}). Then%
\begin{eqnarray*}
dX_{t}\left( w\right) &=&\left[ f\left( X_{t}\left( w\right) \right)
+g\left( X_{t}\left( w\right) \right) \varkappa \left( X_{t}\left( w\right)
\right) \right] dt+g\left( X_{t}\left( w\right) \right) \left[ dw\left(
t\right) -\varkappa \left( X_{t}\left( w\right) \right) dt\right] \\
&=&\left[ f\left( X_{t}\left( w\right) \right) +g\left( X_{t}\left( w\right)
\right) \varkappa \left( X_{t}\left( w\right) \right) \right] dt+g\left(
X_{t}\left( w\right) \right) d\widetilde{w}\left( t\right)
\end{eqnarray*}%
where $\widetilde{w}\left( t\right) $ (from Girsanov theorem) is a Wiener
process on a "new" space $\left( \Omega ,\mathfrak{F},\widetilde{\mathbb{P}}%
\right) $ with a measure%
\begin{eqnarray*}
\widetilde{\mathbb{P}}\left( A\right) &=&\int_{A}\Lambda d\mathbb{P}\text{,
\ }A\in \mathfrak{F} \\
\Lambda &=&\exp \left[ \int_{0}^{T}\varkappa \left( X_{t}\left( w\right)
\right) dw\left( t\right) -\frac{1}{2}\int_{0}^{T}\varkappa ^{2}\left(
X_{t}\left( w\right) \right) dt\right]
\end{eqnarray*}%
($\mathbb{E}_{\mathbb{P}}\left[ \Lambda \right] =1$, because $\varkappa $ is
bounded). It follows that $X_{t}\left( w\right) $ satisfies on $\left(
\Omega ,\mathfrak{F},\widetilde{\mathbb{P}}\right) $ the equation%
\begin{eqnarray}
&&dX_{t}\left( w\right)  \notag \\
&=&\left[ f\left( X_{t}\left( w\right) \right) +g\left( X_{t}\left( w\right)
\right) \left[ \frac{g^{\prime }\left( X_{t}\left( w\right) \right) }{2}-%
\frac{f\left( X_{t}\left( w\right) \right) }{g\left( X_{t}\left( w\right)
\right) }\right] \right] dt+g\left( X_{t}\left( w\right) \right) d\widetilde{%
w}\left( t\right)  \notag \\
&=&g\left( X_{t}\left( w\right) \right) \left[ \frac{g^{\prime }\left(
X_{t}\left( w\right) \right) }{2}dt+d\widetilde{w}\left( t\right) \right]
\label{7}
\end{eqnarray}%
It can be verified directly, that%
\begin{equation}
X_{t}\left( w\right) =c\left( \widetilde{w}\left( t\right) \right)  \label{8}
\end{equation}%
solves (\ref{7}). Hence, we get 
\begin{equation*}
X_{t}\left( w\right) =c\left( \widetilde{w}\left( t\right) \right) =c\left(
w\left( t\right) -\int_{0}^{t}\varkappa \left( X_{s}\left( w\right) \right)
ds\right)
\end{equation*}%
on the "old" space $\left( \Omega ,\mathfrak{F},\mathbb{P}\right) $.
\end{proof}

\begin{example}
(continued). With $g$ and $f$ as above, we have the integral equation, case
(a)%
\begin{equation*}
X_{t}\left( w\right) =\sinh \left( a+w\left( t\right) -\int_{0}^{t}\phi
\left( X_{s}\left( w\right) \right) ds\right)
\end{equation*}%
and Ito's equation%
\begin{equation*}
dx\left( t\right) =\left[ \frac{x\left( t\right) }{2}-\phi \left( x\left(
t\right) \right) \sqrt{1+x^{2}\left( t\right) }\right] dt+\sqrt{%
1+x^{2}\left( t\right) }dw\left( t\right)
\end{equation*}%
case (b)%
\begin{eqnarray*}
X_{t}\left( w\right) &=&\left[ sign\left( a+w\left( t\right)
-\int_{0}^{t}\phi \left( X_{s}\left( w\right) \right) ds\right) \right] \\
&&\times \left[ \left( 1-\alpha \right) \left\vert a+w\left( t\right)
-\int_{0}^{t}\phi \left( X_{s}\left( w\right) \right) ds\right\vert \right]
^{1/1-\alpha }
\end{eqnarray*}%
and Ito's equation%
\begin{equation*}
dx\left( t\right) =\left[ \frac{\alpha }{2}sign\left( x\left( t\right)
\right) \left\vert x\left( t\right) \right\vert ^{2\alpha -1}-\left\vert
x\left( t\right) \right\vert ^{\alpha }\phi \left( x\left( t\right) \right) %
\right] dt+\left\vert x\left( t\right) \right\vert ^{\alpha }dw\left(
t\right)
\end{equation*}
\end{example}

\begin{corollary}
Under the conditions of the equivalence theorem, we have 
\begin{eqnarray}
d\widetilde{w}\left( t\right) &=&-\varkappa \circ c\left( \widetilde{w}%
\left( t\right) \right) dt+dw\left( t\right)  \label{81} \\
\widetilde{w}\left( 0\right) &=&0  \notag
\end{eqnarray}
\end{corollary}

\begin{proof}
From (\ref{4}) follows that%
\begin{eqnarray*}
\widetilde{w}\left( t\right) &=&w\left( t\right) -\int_{0}^{t}\varkappa
\left( X_{s}\left( w\right) \right) ds \\
&=&w\left( t\right) -\int_{0}^{t}\varkappa \circ c\left( w\left( s\right)
-\int_{0}^{s}\varkappa \left( X_{u}\left( w\right) \right) du\right) ds \\
&=&w\left( t\right) -\int_{0}^{t}\varkappa \circ c\left( \widetilde{w}\left(
s\right) \right) ds
\end{eqnarray*}
\end{proof}

\begin{remark}
From (\ref{8}) and (\ref{81}) we have on $\left( \Omega ,\mathfrak{F},%
\widetilde{\mathbb{P}}\right) $%
\begin{equation*}
X_{t}\left( \widetilde{w}+\int_{0}^{\cdot }\varkappa \circ c\left( 
\widetilde{w}\left( s\right) \right) ds\right) =c\left( \widetilde{w}\left(
t\right) \right)
\end{equation*}
\end{remark}

\begin{example}
(a) Since in our example $\varkappa \circ c\left( x\right) =\phi \left(
\sinh \left( a+x\right) \right) $, hence%
\begin{equation*}
d\widetilde{w}\left( t\right) =-\phi \left( \sinh \left( a+\widetilde{w}%
\left( t\right) \right) \right) +dw\left( t\right)
\end{equation*}%
(b) here $\varkappa \circ c\left( x\right) =\phi \left( \left[ sign\left(
a+x\right) \right] \left[ \left( 1-\alpha \right) \left\vert a+x\right\vert %
\right] ^{1/1-\alpha }\right) $, hence%
\begin{equation*}
d\widetilde{w}\left( t\right) =-\phi \left( \left[ sign\left( a+\widetilde{w}%
\left( t\right) \right) \right] \left[ \left( 1-\alpha \right) \left\vert a+%
\widetilde{w}\left( t\right) \right\vert \right] ^{1/1-\alpha }\right)
+dw\left( t\right)
\end{equation*}
\end{example}

\begin{corollary}
(weak solutions) Let $b\left( t\right) $, $t\in \left[ 0,T\right] $ be a
Brownian motions on some probability space $\left( \Omega ^{\prime },%
\mathfrak{F}^{\prime },\mathbb{P}^{\prime }\right) $. Define%
\begin{equation*}
\Lambda =\exp \left[ -\int_{0}^{T}\varkappa \circ c\left( b\left( t\right)
\right) db\left( t\right) -\frac{1}{2}\int_{0}^{T}\left( \varkappa \circ
c\right) ^{2}\left( b\left( t\right) \right) dt\right]
\end{equation*}%
If 
\begin{equation*}
\mathbb{E}_{\mathbb{P}}\Lambda =1
\end{equation*}%
then 
\begin{equation*}
x_{b}\left( t\right) \triangleq c\left( b\left( t\right) \right)
\end{equation*}%
is a (weak) solution of (\ref{4}).
\end{corollary}

\begin{proof}
According to Girsanov theorem 
\begin{equation*}
\mathbb{P}\left( A\right) \triangleq \int_{A}\Lambda d\mathbb{P}^{\prime }%
\text{, \ }A\in \mathfrak{F}
\end{equation*}%
is a probability measure, $\left( \Omega ,\mathfrak{F},\mathbb{P}\right) $
is a probability space, and%
\begin{equation*}
w\left( t\right) \triangleq b\left( t\right) +\int_{0}^{t}\varkappa \circ
c\left( b\left( s\right) \right) ds
\end{equation*}%
is a Wiener process on it. Hence%
\begin{eqnarray*}
x_{b}\left( t\right) &\triangleq &c\left( b\left( t\right) \right) \\
&=&c\left( w\left( t\right) -\int_{0}^{t}\varkappa \circ c\left( b\left(
s\right) \right) ds\right) \\
&=&c\left( w\left( t\right) -\int_{0}^{t}\varkappa \left( x_{b}\left(
s\right) \right) ds\right)
\end{eqnarray*}%
on $\left( \Omega ,\mathfrak{F},\mathbb{P}\right) $.
\end{proof}

\begin{remark}
K$-$ diffusions starting from random initial conditions can be easily
obtained. Let $\xi $ is a random variable on $\left( \Omega ,\mathfrak{F},%
\mathbb{P}\right) $, and consider the following generalization of (\ref{4})%
\begin{equation}
c\left( \xi +w\left( t\right) -\int_{0}^{t}\varkappa \left( X_{s}\left(
w\right) \right) ds\right) =X_{t}\left( w\right)  \label{9}
\end{equation}%
If $\xi $\ is stochastically independent on $w\left( t\right) $, $t\in \left[
0,T\right] $, then the solution of (\ref{9}) is a K$-$ diffusions with $%
X_{0}\left( w\right) =c\left( \xi \right) $.
\end{remark}

\section{Applications}

\subsection{Identification of financial instruments}

Consider two financial instruments. Denote their prices by $X$ and $Y$.
Moreover, assume that $X$ and $Y$ are driven by the same Wiener process and
assume $X$ is a K$-$ diffusion with $c$ and $\varkappa $ known. How can one
identify $Y$? There is a well known method of a "black box" identification
by Norbert Wiener. However, his method is essentially restricted to systems
of special kind; input and output must be observable. This is not the case
in financial modelling. Here we have the black box $w\rightarrow \left(
X_{t}\left( w\right) ,Y_{t}\left( w\right) \right) $ and one may observe the
output only. Hence, this method cannot be applied directly. To overcome this
difficulty, observe that, if $c^{-1}$ exists, than the mapping $w\rightarrow
X_{t}\left( w\right) $ is invertible, and%
\begin{equation*}
w\left( t\right) =c^{-1}\left( X\left( t\right) \right)
-\int_{0}^{t}\varkappa \left( X_{s}\right) ds
\end{equation*}%
is a Wiener process, hence, the input and output of this black box%
\begin{equation*}
X\rightarrow Y_{t}\left( X\right) =Y_{t}\left( c^{-1}\left( X\left( t\right)
\right) -\int_{0}^{t}\varkappa \left( X_{s}\right) ds\right)
\end{equation*}%
is observable. Now, Wiener's method of nonlinear systems identification can
be applied to the $Y-$ black box (see [11], Lecture 10 and 11)

\subsection{Fractional diffusions}

Let $H\in \left( 0,1\right) $ be a Hurst index, $B^{H}\left( t\right) $, $%
t\in \left[ 0,T\right] $ denotes a fractional Brownian motion (fBm) and
define%
\begin{equation*}
\varkappa ^{H}\left( t,x\right) =Ht^{2H-1}g^{\prime }\left( x\right) -\frac{%
f\left( x\right) }{g\left( x\right) }
\end{equation*}

\begin{proposition}
If the mapping $X_{t}^{H}\left( w\right) :\left[ 0,T\right] \times C\left( %
\left[ 0,T\right] ;\mathbb{R}\right) \rightarrow \mathbb{R}$, solves
nonlinear integral equation 
\begin{equation*}
c\left( B^{H}\left( t\right) -\int_{o}^{t}\varkappa ^{H}\left( s,X_{s}\left(
B^{H}\right) \right) ds\right) =X_{t}\left( B^{H}\right)
\end{equation*}%
then it solves (strongly) SDE%
\begin{equation*}
x\left( t\right) =c\left( 0\right) +\int_{0}^{t}f\left( x\left( s\right)
\right) ds+\int_{0}^{t}g\left( x\left( s\right) \right) dB^{H}\left( s\right)
\end{equation*}%
where the stochastic integral is the WIS\ integral.
\end{proposition}

\begin{proof}
Set 
\begin{equation*}
w^{H}\left( t\right) =B^{H}\left( t\right) -\int_{o}^{t}\varkappa ^{H}\left(
s,X_{s}^{H}\left( B^{H}\right) \right) ds
\end{equation*}%
From Ito's formula for fBm (p. 161 of [9]) we have%
\begin{eqnarray*}
&&dc\left( w^{H}\left( t\right) \right) \\
&=&\left[ Ht^{2H-1}c^{\prime \prime }\left( w^{H}\left( t\right) \right)
-\varkappa ^{H}\left( t,X_{t}^{H}\left( B^{H}\right) \right) c^{\prime
}\left( w^{H}\left( t\right) \right) \right] dt+c^{\prime }\left(
w^{H}\left( t\right) \right) dB^{H}\left( t\right) \\
&=&\left[ Ht^{2H-1}gg^{\prime }\left( c\left( w^{H}\left( t\right) \right)
\right) -\left[ Ht^{2H-1}g^{\prime }\left( X_{t}^{H}\left( B^{H}\right)
\right) -\frac{f\left( X_{t}^{H}\left( B^{H}\right) \right) }{g\left(
X_{t}^{H}\left( B^{H}\right) \right) }\right] g\left( c\left( w^{H}\left(
t\right) \right) \right) \right] dt \\
&&+g\left( c\left( w^{H}\left( t\right) \right) \right) dB^{H}\left( t\right)
\\
&=&f\left( X_{t}^{H}\left( B^{H}\right) \right) dt+g\left( X_{t}^{H}\left(
B^{H}\right) \right) dB^{H}\left( t\right)
\end{eqnarray*}%
since by the assumption $c\left( w^{H}\left( t\right) \right)
=X_{t}^{H}\left( B^{H}\right) $.
\end{proof}

\subsection{Smooth densities}

Set $\widetilde{F}\left( x\right) =\mathbb{P}\left( \widetilde{w}\left(
t\right) <x\right) $. Then%
\begin{eqnarray*}
\mathbb{P}\left( X_{t}\left( w\right) <x\right) &=&\mathbb{P}\left( c\left( 
\widetilde{w}\left( t\right) \right) <x\right) \\
&=&\widetilde{F}\circ c^{-1}\left( x\right)
\end{eqnarray*}%
Hence, the smoothness density problem for $X_{t}\left( w\right) $, is
reduced to investigation of ordinary function $\widetilde{F}\circ c^{-1}$.

\begin{example}
\begin{eqnarray*}
\text{(a) \ }\widetilde{F}\circ c^{-1}\left( x\right) &=&\widetilde{F}\left(
\sinh ^{-1}\left( a+x\right) \right) \\
\text{(b) \ }\widetilde{F}\circ c^{-1}\left( x\right) &=&\widetilde{F}\left(
sign\left( a+x\right) \left( 1-\alpha \right) ^{1-\alpha }\left\vert
a+x\right\vert ^{1-\alpha }\right)
\end{eqnarray*}
\end{example}

\begin{acknowledgement}
I would like to thank Professor Moshe Zakai for his remarks and suggestions.
\end{acknowledgement}

\begin{center}
\bigskip {\Large References}
\end{center}

[1] A.N. Kolmogorov, Uber die analytischen methoden in der

\ \ \ \ wahrscheinlichkeitsrechnung, Math. Ann. 104, (1931), 415-458

[2] K. Ito, On stochastic differential equations, Memoirs Amer.

\ \ \ \ Math. Soc. 4, (1951), 1- 51

[3] D.W. Stroock, S.R.S. Varadhan, Multidimensional Diffusion Processes,

\ \ \ \ Springer-Verlag 1979

[4] E. Isobe, S. Sato, Wiener-Hermite expansion of a process generated

\ \ \ \ by an Ito stochastic differential equation, J. Appl. Prob. 20
(1983), 754-765

[5] K. Ito, Selected Papers, Edited by D.W. Stroock and S.R.S. Varadhan,

\ \ \ \ Springer-Verlag 1987

[6] D.W. Stroock, Markov Processes from K. Ito's Perspective,

\ \ \ \ Princeton Univ. Press 2003

[7] H. Sussman, An interpretation of stochastic differential equations

\ \ \ \ as orinary differential equations which depend on the sample point,

\ \ \ \ Bull.Amer. Math. Soc. 83, (1977) 296-298

[8] H. Doss, Liens entre equations differentielles stochastiques et
ordinaires,

\ \ \ \ Ann. Inst. Henri Poincare 13, (1977), 99-125 \ F.

[9] F. Biagini, Y. Hu, B. Oksendal, T. Zhang, Stochastic Calculus for
Fractional

\ \ \ \ Brownian Motion and Applications, Springer, 2008

[10] I. Karatzas, S.E. Shreve, Brownian Motion and Stochastic Calculus,

\ \ \ \ \ Springer-Verlag, 1991

[11] N. Wiener, Nonlinear Problems in Random Theory, Tech. Press of MIT,

\ \ \ \ \ \ John Wiley\&Sons, Chapman\&Hall, 1958

\end{document}